\documentclass[a4paper]{article}
\usepackage{graphicx}
\usepackage{amsfonts}
\usepackage{amsmath}
\usepackage{amssymb}
\usepackage{fancyhdr}
\usepackage{titlesec}
\usepackage{indentfirst}
\usepackage{booktabs}
\usepackage{verbatim}
\usepackage{color}
\usepackage{amsthm}
\usepackage{subfigure}
\usepackage{abstract}
\usepackage{fullpage}
\usepackage{hyperref}
\usepackage[margin=0pt]{geometry}
\usepackage{lipsum}
\usepackage{layout}
\usepackage{mathtools}
\usepackage[page,header]{appendix}
\usepackage{titletoc}
\numberwithin{equation}{section}
\newtheorem{theorem}{Theorem}[section]
\newtheorem{lemma}[theorem]{Lemma}
\newtheorem{corollary}[theorem]{Corollary}

\newtheorem{conjecture}{Conjecture}

\newcommand{\RNum}[1]{\uppercase\expandafter{\romannumeral #1\relax}}

\DeclareMathOperator{\disc}{disc}
\DeclareMathOperator{\Gal}{Gal}
\DeclareMathOperator{\ind}{ind}

\DeclareMathOperator{\Disc}{Disc}
\DeclareMathOperator{\Nm}{Nm}

\begin{document}

\title{Secondary Term of Asymptotic Distribution of $S_3\times A$ Extensions over $\mathbb{Q}$}
\author{Jiuya Wang}
\newcommand{\Addresses}{{
		\bigskip
		\footnotesize		
		Jiuya Wang, \textsc{Department of Mathematics, University of Wisconsin-Madison, 480 Lincoln Dr., Madison, WI 53706, USA
		}\par\nopagebreak
		\textit{E-mail address}: \texttt{jiuyawang@math.wisc.edu}
	}}
\maketitle	
	\begin{abstract}	
           We combine a sieve method together with good uniformity estimates to prove a secondary term for the asymptotic estimate of $S_3\times A$ extensions over $\mathbb{Q}$ when $A$ is an odd abelian group with minimal prime divisor greater than $5$. At the same time, we prove the existence of a power saving error when $A$ is any odd abelian group.
	\end{abstract}
	
\bf Key words. \normalfont Malle's conjecture, secondary term, uniformity estimates, power saving error.
\pagenumbering{arabic}	
\section{Introduction}
Davenport and Heilbronn \cite{DH71} proved a celebrated theorem on the asymptotic distribution of $S_3$ cubic extensions over $\mathbb{Q}$, which gives the average $3$-class number of quadratic extensions. This is the first and only proven case for Cohen-Lenstra heuristics over quadratic fields up till now. This work was generalized by Datskovsky and Wright \cite{DW88} to any global field with characteristic not equal to $2$ or $3$, along with the average class number result. Given $G\subset S_n$ a permutation group, denote $N_k(G, X)$ to be the number of extensions over $k$ with $\Gal(K/k)$ isomorphic to $G$ as a permutation group and with the absolute discriminant bounded by $X$. Then their results state as following:
\begin{theorem}[\cite{DH71,DW88}]\label{S3}
	There exist a constant $C$ such that 
	$$ N_k(S_3, X)  \sim CX,$$
	where $k$ is any global field with characteristic not equal to $2$ or $3$. 
\end{theorem}
What is more striking about this counting $N_{\mathbb{Q}}(S_3, X)$ is that it has a secondary term in the order of $X^{5/6}$. The existence of this secondary term, called Roberts' conjecture, is conjectured in both \cite{DW88} and \cite{Rob01}. This conjecture was proved independently by Bhargava, Shankar and Tsimerman \cite{BST} and by Taniguchi and Thorne \cite{TT13} at the same time, but with very different methods. A secondary term for the average class number is also proved in both papers. By combination of these two methods, Bhargava, Taniguchi and Thorne \cite{BTT} are able to prove this result with a better error term. Moreover in both \cite{TT13} and \cite{BTT}, the asymptotic distributions of $S_3$ cubic extensions with local conditions are obtained with an explicit dependency of local parameter in the error term, which our paper heavily depends on. 
\begin{theorem}[\cite{BTT}, Theorem $4.3$]\label{S3BTT}
	There exists constant $A$ and $B$ such that the asymptotic distribution of $S_3$ cubic extensions over $\mathbb{Q}$ is
	$$N_{\mathbb{Q}}(S_3, X) = AX + B X^{5/6} + O(X^{2/3+\epsilon}).$$
\end{theorem}
Malle brought forward his conjecture \cite{Mal02} and \cite{Mal04} on the order of the main term in $N_k(G, X)$, however there is still little understanding towards the secondary term. Thorne has a summary \cite{Th12} on all approaches to understand the secondary term for cubic fields, including Hough's \cite{Hou} and Zhao's work \cite{Zhao} on variations of Roberts' conjecture from different perspective, aside from the results mentioned above. 

It is surely beneficial if more examples of secondary terms for asymptotic estimates of  $N_k(G, X)$ are presented. It is natural to look at the asymptotic distribution of $S_4$ quartic fields $N_{\mathbb{Q}}(S_4, X)$, of which the main term is proved in \cite{Bha05}, since they are also parametrized by orbits in a pre-homogeneous vector space and give average $2$-class number of $S_3$ cubic fields. In \cite{Co06}, the authors record a conjectural secondary term in the order of $X^{5/6}$ of quartic fields by Yukie, along with a third term in the order of $X^{3/4}\ln X$ and even a fourth term $X^{3/4}$. However no proof on the secondary term in the quartic case is known. On the other hand, Taniguchi and Thorne \cite{TT14} conjectured a secondary term with precise constant on $S_3$ sextic fields, and a third term is also conjectured. It would be possible to prove the secondary term in the sextic case if both the exponent of $X$ and the dependency of the local parameters could be improved a lot in the error term of the distribution of cubic fields with local conditions. 

In a recent paper of the author \cite{JW17}, the asymptotic distribution of $S_3\times A$  fields are obtained in terms of a precise main term when $A$ is any odd abelian group. The main result of our paper is to prove the secondary term for the asymptotic distribution of $S_3\times A$ number fields with degree $3|A|$ for $A$ with minimal prime divisor greater than $5$. This provides a second example of a secondary term in distribution of number fields, and actually infinitely many such examples.
%
\begin{theorem}\label{Thm1}
	Let $A$ be an abelian group with minimal prime divisor greater than 5. Then there exist $C_1$, $C_2$ and $\delta>0$ such that the asymptotic distribution of $S_3\times A$ number fields with degree $3|A|$ over $\mathbb{Q}$ by absolute discriminant is $$N_{\mathbb{Q}}(S_3\times A, X) = C_1 X^{1/|A|} + C_2 X^{5/6|A|} + O(X^{5/6|A| - \delta}).$$
\end{theorem}
The constants $C_1$ and $C_2$ are all finite sum of Euler products. As an example, in section $4.7$ we give the precise constants $C_1$ and $C_2$ when $A = C_l$ is cyclic group with prime order $l>5$.

For $A$ with minimal prime divisor $3$ or $5$, we prove a weaker result, i.e., a power saving error is obtained. 
\begin{theorem}\label{Thm2}
	Let $A$ be any odd abelian group. Then there exist $C$ and $\delta>0$ such that the asymptotic distribution of $S_3\times A$-number fields over $\mathbb{Q}$ by absolute discriminant is $$N_{\mathbb{Q}}(S_3\times A, X) = CX^{1/|A|} + O(X^{1/|A| -\delta}).$$
\end{theorem}
The amount of power saving $\delta$ in both Theorem \ref{Thm1} and \ref{Thm2} are computed in section \ref{sec48}. 

To prove these results, we apply a sieve method for the range of small primes, building on the distribution of cubic fields with local conditions, and we prove a new uniformity estimate of ramified cubic fields for the range of large primes. For these cases in Theorem \ref{Thm1} and \ref{Thm2}, this method provides another method to prove Malle's conjecture. However, the method in \cite{JW17} does not require any information about the error from both $S_n$ extensions for $n = 3,4,5$ and $A$ extensions, therefore we could get the main term proven in more cases there. Neither method subsumes the other.

Based on these new examples, we give the following conjecture.
\begin{conjecture}
	Given $G_i\subset S_{n_i}$ for $i = 1, 2$,  if the asymptotic distribution of $G_1$
	extensions has a secondary term in the order of $X^{c_1}$ such that $\ind(G_2)> \frac{n_2}{n_1 c_1}$, then the asymptotic distribution of $G_1\times G_2$ extensions has a secondary term in the order of $X^{c_1/n_2}$.
\end{conjecture}

We organize the paper as following. In section $2$, we give a summary on necessary results as preliminaries. This includes the description of discriminant of the compositum of two disjoint number fields, and the product lemma. In section $3$, we prove a new uniformity estimate on partially ramified cubic fields by geometric sieve. In section $4$, we apply the sieve method to prove the main theorem. \\

%

\textbf{Notations}\\
$p$: a finite place in base field $k$ or a prime number\\
$|\cdot|$:  absolute norm $\Nm_{k/\mathbb{Q}}$\\
$\Disc(K)$:  absolute norm of $\disc(K/\mathbb{Q})$\\
$\Disc_p(K)$: $p$ part of $\Disc(K)$\\
$\tilde{K}$: Galois closure of $K$ over $\mathbb{Q}$\\
$\ind(\cdot)$: the index $n$ - $\sharp \{\text{orbits}\}$ for a cycle or minimum value of index among non-identity elements for a group\\
$N_k(G,X)$: the number of isomorphic classes of $G$-extension over $k$ with $\Disc$ bounded by $X$

\section{Preliminaries}
In this section, we will give a summary on the discussion of the discriminant of compositum and the product argument in the author's previous paper \cite{JW17}. 
\subsection{Discriminant of $S_3\times A$}
Throughout the section, $K$ and $L$ are finite field extensions but not necessarily Galois. We will denote $\Gal(K/\mathbb{Q})$ to be the Galois group $\Gal(\tilde{K}/\mathbb{Q})\subset S_{n}$ as a permutation subgroup which acts on the $n$ embeddings of $K$ into $\bar{\mathbb{Q}}$ where $n = \deg(K/\mathbb{Q})$. 

Given a pair of extensions $(K, L)$ over $\mathbb{Q}$ where $\Gal(K) = G_1\subset S_n $ and $\Gal(L) = G_2 \subset S_{m}$ which intersect trivially, our goal is to determine the discriminant $\Disc(KL)$ completely. If we know the ramification information of $K$ and $L$ completely, we would be able to pin down the absolute discriminant of the compositum $\Disc(KL)$. Indeed, isomorphism classes of \'etale extensions of degree $n$ over $\mathbb{Q}$ (or $\mathbb{Q}_p$) are in one-to-one correspondence to $G_{\mathbb{Q}} \to S_n$ ($G_{\mathbb{Q}_p} \to S_n$ ) up to relabeling the letters, therefore we get the map $\phi_1: G_{\mathbb{Q}_p}\hookrightarrow G_{\mathbb{Q}} \to S_n$ for $K$ and similarly $\phi_2$ for $L$ that record the complete local information. The direct product of $\phi_1\times \phi_2: G_{\mathbb{Q}_p}\to S_{n}\times S_m\subset S_{mn}$ corresponds to the local \'etale extension $(KL)_p : = (KL) \otimes_{\mathbb{Q}} \mathbb{Q}_p$ with $\Disc((KL)_p/\mathbb{Q}_p) = \Disc_p(KL)$. Therefore the data of $\phi_1$ and $\phi_2$ is sufficient to determine the discriminant. 

Moreover at a prime $p$ where both $K$ and $L$ are tamely ramified, we can determine $\Disc_p(KL)$ with less data. Such primes are all but finitely many. In this case, the exponent of $p$ in the discriminant could be determined by the inertia group as a permutation subgroup. A tamely ramified inertia group could be generated by a single element, so let's denote the inertia group of $K$ and $L$ by $I_1 = \large \langle g_1 \large \rangle $ and $I_2 = \large \langle g_2\large \rangle $. 
\begin{theorem}[\cite{JW17}, Theorem $2.2$, $2.3$]\label{indge}
	Let $K$ and $L$ be as given above which are both tamely ramified at $p$. Let $g_i$ for $i = 1,2$, with order $e_i$, be the generator of inertia group at $p$ for $\tilde{K}$and $\tilde{L}$. The generator $g_1$ is a product of $k$ disjoint cycles $\prod c_k$ and $g_2$ is a product of $l$ disjoint cycles $\prod d_l$. Then the exponent for $p$ in $\Disc_p(KL)$ is $mn-\sum_{k,l} \gcd(|c_k|, |d_l|)$. Moreover, if $(e_1, e_2) = 1$, then the exponent for $p$ in $\Disc_p(KL)$ is $mn-kl$.
\end{theorem}

At wildly ramified primes, in most cases we do not have a direct way to compute, but there are only finitely many extensions over $\mathbb{Q}_p$ with bounded degree, so there are only finitely many possibilities for $\phi_1$ and $\phi_2$ for a certain pair of $(G_1, G_2)$. It will be shown in the discussion that they only affect the coefficient of the main term and the secondary term. 

\subsection{Product Argument}
In this section, we are going to include the product lemma on two distributions with different order of growth. We will apply this lemma on tail estimates. Denote $F_i$, $i = 1, 2$,  to be asymptotic distribution of some multi-set of positive integers $S_i$, i.e., $F_i(X) = \sharp\{ s\in S_i \mid s\le X\}$, and denote the product distribution $P_{a,b}(X) = \sharp\{ (s_1,s_2)\mid s_i\in S_i, s_1^as_2^b\le X\} $ where $a, b>0$.
\begin{lemma}[\cite{JW17}, Lemma $3.2$]\label{proneq}
	Let $F_i(X)$, $i=1,2$ be as given above and $F_i(X) \sim A_i X^{n_i}\ln^{r_i} X$ where $0 <n_i\le 1$ and $r_i\in \mathbb{Z}_{\ge 0}$.  If 
	$\frac{n_1}{a}- \frac{n_2}{b} > 0$, then there exists a constant $C$ such that  $$P_{a,b}(X)\sim C X^{\frac{n_1}{a}}\ln^{r_1} X.$$
	Furthermore if $F_i(X) \le A_i X^{n_i}\ln^{r_i} X$, then we have
	$$P_{a,b}(X)\le A_1A_2 \frac{r_2!}{b^{r_2}a^{r_1}} \frac{1}{(\frac{n_1}{a}-\frac{n_2}{b})^{r_2+1}}\frac{n_1}{a}X^{\frac{n_1}{a}} \ln^{r_1} X.$$
\end{lemma}

\section{Uniformity}
In this section we prove a new uniformity result on partially ramified $S_3$ cubic extensions at finitely many primes and merge this uniformity result with previous known uniformity estimates on totally ramified $S_3$ cubic fields. Although we only need these results over $\mathbb{Q}$, all of the results hold over arbitrary number field $k$. 

Let $k$ be a number field and $q$ be a square-free integral ideal in $\mathcal{O}_k$. Let us deonte $N_{q, r}(S_3, X)$ to be the number of $S_3$ cubic extensions over $k$ that are partially ramified at all places $p|q$, and totally ramified at all places $p|r$. Then we have Proposition 6.2 from \cite{DW88}:
  \begin{theorem}[\cite{DW88}, Proposition $6.2$]\label{unito}
  	The number of non-cyclic cubic extensions over $k$ which are totally ramified at a product of finite places $r = \prod{p_i}$ is:
  	$$N_{1, r}(S_3,X) = O(\frac{X}{|r|^{2-\epsilon}}),$$
  	for any number field $k$ and any square-free integral ideal $r$. The constant is independent of $q$, and only depends on $k$. 
  \end{theorem}      	
  On the other hand, by an argument in the author's previous work \cite{JW17} based on the geometric sieve method introduced in \cite{Bhasieve}, we get the following uniformity estimates on partially ramified extensions.
   \begin{theorem}\label{unipa}
   	The number of non-cyclic cubic extensions over $k$ which are partially ramified at a product of finite places $q = \prod{p_i}$ is:
   	$$N_{q,1}(S_3,X) = O(\frac{X}{|q|^{1/6-\epsilon}}),$$
   	for any number field $k$ and any square-free integral ideal $q$. The constant is independent of $q$, and only depends on $k$. 
   \end{theorem}      	
  This result comes from Theorem $4.5$ in \cite{JW17} and the observation that if we just focus on the number of cubic orders ramified at a fixed finite set of places, then we can improve the power saving error in the geometric sieve\cite{Bhasieve} and therefore drop the codimension $2$ condition.  We could similarly get the uniformity result for ramified $S_4$ and $S_5$ extensions by the same way. As a corollary of Theorem \ref{unipa}, we get the corresponding estimates on the average $3$-class number over quadratic fields ramified at $q = \prod p_i$. Given $F$ a quadratic extension over $k$, denote $h_3^{*}(F/k)$ to be the relative $3$-class number of $F$ over $k$.
   \begin{corollary}\label{unih}
   	Given a square-free integral ideal $q$, the $3$-class number summed over quadratic extensions $F/k$ with $q| \disc(F/k)$ is bounded by
   	$$\sum_{\substack{[F:k = 2]\\ q| \disc(F), \Disc(F) \le X}} h_3^*(F/k) = O(\frac{X}{|q|^{1/6-\epsilon}}).$$
   \end{corollary}
   \begin{proof}   	
    By \cite{DW88}, there is a one-to-one correspondence between the unramified abelian cubic extensions $L/F$ such that the resulting Galois group of $\tilde{L}/k$ is $S_3$ and the isomorphism classes of nowhere totally ramified non-cyclic cubic extensions $K_3/k$. Moreover, in this correspondence, we have $\Disc(F) = \Disc(K_3)$. If $q| \disc(F)$, then the cubic field $K_3$ is partially ramified at $q$. Therefore
   	\begin{equation}
   	\begin{aligned}
   	\sum_{\substack{[F:k= 2]\\ q| \disc(F), \Disc(F) \le X}} \frac{h_3^*(F/k)-1 }{2}= O(\frac{X}{|q|^{1/6-\epsilon}}).
   	\end{aligned}
   	\end{equation} 
   	Indeed the left-hand side corresponds to the number of nowhere totally ramified $S_3$ cubic extensions which are partially ramified at $q$, and it is a subset of $S_3$ cubic extensions that are partially ramified at $q$. The right-hand side gives the upper bound on this number by Theorem  \ref{unipa}. Rearranging the expression, and applying Theorem $4.2$ \cite{JW17} on quadratic extensions
   	$$ \sum_{\substack{[F:k = 2]\\ q| \disc(F), \Disc(F) \le X}}  1 = O(\frac{X}{|q|^{1-\epsilon}}),$$
   	we have that 
   	$$\sum_{\substack{[F:k = 2]\\q| \disc(F), \Disc(F) \le X}} h_3^*(F/k) = O(\frac{X}{|q|^{1/6-\epsilon}}) + O(\frac{X}{|q|^{1-\epsilon}}) = O(\frac{X}{|q|^{1/6-\epsilon}}).$$       	 
   \end{proof}   
   And by combining the Theorem \ref{unito} and \ref{unipa} using class field theory, we prove the following theorem. 
   \begin{theorem}\label{unipato}
         	The number of non-cyclic cubic extensions over $k$ that are partially ramified at $q = \prod {p_i}$ and totally ramified at $r= \prod {p_j}$ is bounded by 
         	$$N_{q,r}(S_3, X) = O(\frac{X}{|q|^{1/6-\epsilon}|r|^{2-\epsilon}}),$$
         	for any number field $k$ and any square-free integral ideal $qr$. The constant is independent of $q$ and $r$, and only depends on $k$. 
   \end{theorem} 
   \begin{proof}
   		Let $F$ be a quadratic extension over $k$ and $q$ be an integral ideal that divides $\disc(F)$. Let $f$ be an integral ideal in $k$ and denote the conductor of an abelian cubic extension of $F$. We would like to count $S_3$ extensions that are partially ramified at $q$, so it suffices to look at quadratic fields $F$ with $q|\disc(F)$. We would also like to count $S_3$ extensions that are totally ramified at $r$, so it suffices to look at cubic abelian extensions over $F$ with conductor divided by $r$. By Lemma $6.2$ \cite{DW88}, the number of cubic extensions over $F$ with conductor $f$ such that the resulting Galois group over $k$ is $S_3$, could be bounded by $O(4^{\omega(f)} h_3^*(F/k))$ where $\omega(f)$ is the number of prime divisors of $f$, and the implied constant only depends on $k$. So we just need to bound
   		\begin{equation}
   		\begin{aligned}
   		&\sum_{\substack{[F:k] = 2\\ q| \disc(F)}} \sum_{\substack{ r|f\\ |f|^2 \Disc(F) \le X}} 4^{\omega(f)} h_3^*(F/k) \\
   		= & 4^{\omega(r)} \sum_{f}  4^{\omega(f)} \sum_{\substack{[F:k] = 2\\ q|\disc(F), \Disc(F) \le \frac{X}{|f|^2|r|^2}}} h_3^*(F/k)\\
   		\le & 4^{\omega(r)} \sum_{f}  4^{\omega(f)} \frac{X}{|f^2r^2| |q|^{1/6-\epsilon}}\\
   		\le & O(\frac{X}{|q|^{1/6-\epsilon}|r|^{2-\epsilon}}) \sum_{f}  \frac{4^{\omega(f)} }{|f|^2} \le O(\frac{X}{|q|^{1/6-\epsilon}|r|^{2-\epsilon}}).
   		\end{aligned}
   		\end{equation}
   	\end{proof}
   	
   	\begin{proof}[\textbf{Proof of Theorem \ref{unipa}}]
   		We will prove over $\mathbb{Q}$, and the result holds equally when the base field $k$ is an arbitrary number field by Theorem $4.7$ in \cite{JW17}.
   		
   		Firstly, recall that cubic orders are parametrized as $\text{GL}_2(\mathbb{Z})$-orbits of the space of binary cubic forms  $V(\mathbb{Z}) = \{ ax^3+ bx^2y+cxy^2 + dy^3 \mid(a,b,c,d)\in \mathbb{Z}^4 \}$. Please see details in section $2$ and $3$ in \cite{BST}. By Theorem $4.5$ in \cite{JW17}, let us denote $Y$ to be the variety that describes the ramification type introduced in \cite{Bhasieve}, we just need to integrate the the following integrand
   			\begin{equation}
   			\begin{aligned}           		
   			L^1 =  \sharp \{x \in mrB\cap V^{(i)}_{\mathbb{Z}} \mid  x (\text{mod } q) \in Y(\mathbb{Z}/ q\mathbb{Z})\} = O(C^{\omega(q)} )\cdot \max\{\frac{\lambda^4}{q}, \lambda^3t^3\},\\ 
   			\end{aligned}
   			\end{equation}
   		over the fundamental domain of $\text{GL}(\mathbb{R})/\text{GL}(\mathbb{Z})$ where $t\ge \sqrt[4]{3}/ \sqrt{2}$. Please see section $5$ in \cite{BST} for more details on the description of the fundamental domain. Let's denote $S$ to be the set of cubic orders that are ramified at $q$, then 
   		\begin{equation}
   		\begin{aligned}
   		N(S; X) &  \le O(C^{\omega(q)} )\frac{1}{M_i} \int^{O(X^{1/4})}_{\lambda = O(1)} \int^{O(\lambda^{1/3})}_{t=\sqrt[4]{3}/ \sqrt{2}} \max \{ \frac{\lambda^4}{q}, \lambda^3t^3 \} t^{-2}\text{d}t^{\times} \text{d}\lambda^{\times} \\
   		& = O(C^{\omega(q)} )\frac{1}{M_i} \int^{O(X^{1/4})}_{\lambda = O(1)} \max \{ \frac{\lambda^4}{q}, \lambda^3\lambda^{1/3} \} \text{d}\lambda^{\times}\\
   		&  = O(C^{\omega(q)} )\cdot  \max \{  \frac{X}{q}, X^{5/6} \} = O(C^{\omega(q)}) \cdot \max \{  \frac{X}{q}, X^{5/6} \}.\\
	    \end{aligned}
   		\end{equation}
     Since $|q|<X$, we have the number bounded by $O(\frac{X}{|q|^{1/6-\epsilon}})$. The global case follows similarly.   		
   	\end{proof}
   	
\section{Main Proof}
In this section we are going to prove Theorem \ref{Thm1} and Theorem \ref{Thm2}. We will give the outline of the proof in section \ref{sec41}. Then we will compute carefully what the error terms are for each step of summation in section \ref{sec42} and \ref{sec43}. In section \ref{sec44} we will determine the tail estimates based on the uniformity estimates. In section \ref{sec45} we will put all the estimates together and balance between the small range and the large range to optimize the exponent of the power saving error. In section \ref{sec46}, we compute the group theory data required as the final input to prove Theorem \ref{Thm1} and \ref{Thm2}. In section \ref{sec47}, as an example, we give the precise expression of the constant in the main term and the secondary term for $S_3\times C_l$ extensions where $l$ is a prime number. In section \ref{sec48}, we describe the amount of power saving away from the secondary term for cases in Theorem \ref{Thm1}, and the amount of power saving away from the main term for cases in Theorem \ref{Thm2}. 

\subsection{Framework}\label{sec41}
In this section, we are going to give a framework of the proof. Let $K$ be an $S_3$ cubic extension over $\mathbb{Q}$, and $L$ be an $A$ extension over $\mathbb{Q}$. Let $T$ be the set of all primes that divide $6|A|$. Define $\Sigma_p$ as follows: if $p\notin T $, let $\Sigma_p$ be the set of all possible non-trivial inertia groups for an $S_3$ cubic extensions up to conjugation; if $p \in T $, then let $\Sigma_p$ be the set of all possible local \'etale extensions over $\mathbb{Q}_p$ for an $S_3$ cubic extension. Similarly, we define $\Lambda_p$ for $A$-extensions at $p\notin T$ and $p\in T$ separately. Therefore define $\mathcal{A}=\{ \large \langle (12) \large \rangle ,  \large \langle (123) \large \rangle   \}$ and $\mathcal{B} = \{  \large \langle a \large \rangle \mid a \ne e \in A  \}$, then $\Sigma_p = \mathcal{A}$ and $\Lambda_p  =  \mathcal{B}$ for $p\notin T $. We will write $K \in \sigma_p$ for a certain $\sigma_p \in \Sigma_p$: if $K_p$ is isomorphic to $\sigma_p$ at $p\in T$, or if $\tilde{K}$ has $\sigma_p$ as the inertia group at $p\notin T$. Similarly for $L$. By the way $\Sigma_p$ and $\Lambda_p$ are defined, all $K\in\sigma_p$ have the same discriminant, $\Disc_p(K)$, so we could denote this number $\Disc(\sigma_p)$. Similarly for $\Disc(\lambda_p)$ for $A$ extensions. 

Given a pair of extensions $(K, L)$ where $\Gal(K) = S_3$ and $\Gal(L) = A$, by section $2.1$ and Theorem \ref{indge} we would be able to determine $\Disc_p(KL)$. At a certain $p$, say $K\in \sigma_p$ and $L\in \lambda_p$, then denote $\Disc(\sigma_p, \lambda_p)$ to be the local discriminant determined by the pair, and define $e(\sigma_p, \lambda_p)$ as
$$ p^{e(\sigma_p, \lambda_p)} = \frac{\Disc(\sigma_p)^m \Disc(\lambda_p)^n}{\Disc(\sigma_p, \lambda_p)}. $$ 
The exponent $e(\sigma_p, \lambda_p)$ for $p\notin T$ could be determine by Theorem \ref{indge}, and in such cases $e(\sigma_p, \lambda_p)$ is independent of $p$ and only depends on the permutation presentation of $\sigma_p$ and $\lambda_p$. 

Denote the set $\mathcal{S} = \mathcal{A}\times \mathcal{B}  = \{ s_{ij}\mid s_{ij} = (a_i, b_j), a_i\in \mathcal{A}, b_j\in \mathcal{B}\}$ to be the direct product of $\mathcal{A}$ and $\mathcal{B}$, and the set $\mathcal{W} = \prod_{p\in T}(\Sigma_p\times \Lambda_p) = \{ w \mid \forall p\in T,  w_p = (\sigma_p, \lambda_p) \in \Sigma_p\times \Lambda_p \}$. Here $\mathcal{S}$ lists all possible ramification types for a pair $(K, L)$ at tamely ramified places, and $\mathcal{W}$ lists all possible local \'etale extensions for a pair at wildly ramified places. Denote $\rho= (w, q_{ij})$ to be one element $w= \prod_{p\in T}(\sigma_p, \lambda_p) \in \mathcal{W}$ and a tuple of square-free numbers $\rho =  (q_{ij})$, where for each $1\le i\le |\mathcal{A}|$ and $1\le j\le |\mathcal{B}|$, and each $p| q_{ij}$ we have $p \notin T$, and $\prod_{i,j}p_{ij}$ is also square-free. For each $\alpha = (f_{ij}) \in( \mathbb{Z}/2\mathbb{Z})^{|\mathcal{S}|}$, we define $(K, L)\in \rho^{\alpha}$ as follows: 1) the pair $(K,L)$ satisfies the condition $w$ at all $p\in T$;  2) at each $p|q_{ij}$, we require $K\in a_i$ and $L\in b_j$; 3) if $f_{ij}= 0$, then we require further that $p| q_{ij}$ are the only primes that $(K, L)$ are simultaneously in $a_i$ and $b_j$.  

Define
$$B(\rho^{\alpha}, X) = \sharp \{ (K, L) \mid (K, L) \in \rho^{\alpha},  \Disc(K)^m\Disc(L)^3 \le X \} ,$$
where $m = |A|$. 
If $\alpha = 0 \in (\mathbb{Z}/2\mathbb{Z})^{|S|}$, then we get for $(K, L )\in \rho^{0}$ that
$$\Disc(KL) = \frac{\Disc(K)^m \Disc(L)^3}{\prod_{p\in T} p^{e(\sigma_p, \lambda_p)}\prod_{i,j} q_{ij}^{e(a_i, b_j)}} = \frac{\Disc(K)^m \Disc(L)^3}{L_{\rho}}.$$
Therefore if we could get an estimation of $B(\rho^{0}, X)$ for every $\rho$, we just need to sum $B(\rho^{0}, XL_{\rho}) $
over all $\rho$ in this form to get the final counting
$$G(X) = \sharp \{ (K, L ) \mid \Disc(KL) \le X  \} = \sum_{\rho} B(\rho^{0}, XL_{\rho}).$$

In order to get $B(\rho^{0}, X)$, we apply a sieve method. We will say that $\rho_1$ \textit{divides} $\rho_2$ if they contain the same $w \in \mathcal{W}$, and for each $i$ and $j$, we have $q^{(1)}_{ij} | q^{(2)}_{ij}$, where $q^{(k)}_{ij}$ is the associated square-free number at the $i,j$-th position in $\rho_k$. Given $\rho_1$ and $\rho_2$ with the same $w \in \mathcal{W}$, we can also multiply to get a new tuple $(\rho_1\rho_2)_{ij} = (q^{(1)}_{ij} q^{(2)}_{ij} )$ when it is legal, i.e., when the product $\prod_{i,j} q^{(1)}_{ij}q^{(2)}_{ij}$ is still square-free. By inclusion-exclusion, we have the following relation
\begin{equation}\label{incexc}
\begin{aligned}
B(\rho^{0}, X) = \sum_{\varrho = \rho \eta} \mu(\eta) B(\varrho^{1}, X),
\end{aligned}
\end{equation}
where we define $\mu(\eta)$ to be $\prod_{i,j} \mu(q_{ij})$ with $q_{ij}$ the $i,j$-th square-free integer in $\eta$. Here we write $\varrho^{1}$ in short for $\varrho^{(1,1, \cdots, 1)}$, which means that we require no condition on places outside $q_{ij}$ in $\varrho$. So we can apply product argument to distributions of $S_3$ cubic extensions and $|A|$-extensions with local conditions to get $B(\varrho^{1}, X)$. When $\rho$ involves some big primes, we will apply uniformity estimates to get a tail estimation on $B(\rho^{0}, X)$ and use that instead.  

\subsection{Estimates of $B(\varrho^1, X)$}\label{sec42}
In this section, we are going to compute the product distribution of $S_3$ cubic extensions and $A$-extensions, in addition with local conditions on ramification. Our computation heavily relies on the following theorem in \cite{BTT}, which improves previous results on distribution of $S_3$ cubic extensions with local density \cite{TT13}. On one hand it reduces the exponent of $X$ in error terms, and on the other hand, it also reduces the dependency of local parameters in the constant of the error terms. 

\begin{theorem}[\cite{BTT}, Theorem $4.3$]\label{S3local}
 The number of $S_3$ cubic extension that are partially ramified at $q = \prod p_i$ and totally ramified at $r = \prod p_j$ are estimated to be 
 $$N_{q,r}(S_3, X) = AA_qA_{r^2} X + B B_q B_{r^2} X^{5/6} + O(C_q C_{r^2} X^{2/3+\epsilon}),$$
 where the constants $A_n$, $B_n$ and $C_n$ are some multiplicative arithmetic functions.
\end{theorem}
We record the above densities in the form as we need them. For a complete table of every local condition, please see $(6.8)$ and page $2487$ in \cite{TT13}. However we only need the local density on ramified cubic fields. In these cases, for each prime number $p$, 
$$A_p = \frac{C_p^{-1}}{p}, \quad A_{p^2} = \frac{C_p^{-1}}{p^2},$$
where $C_p =1+ p^{-1} + p^{-2}$ is the normalizing factor, and
$$B_p = \frac{K_p^{-1} (1 + p^{-1/3})^2}{p}, \quad B_{p^2} =  \frac{K_p^{-1} (1 + p^{-1/3})}{p^2},$$
where $K_p = \frac{(1-p^{-5/3})(1+p^{-1})}{1-p^{-1/3}}$ is the normalizing factor, and
$$C_p = p^{4/5}, \quad C_{p^2} = p^{4/5}.$$
The constants are $A = \frac{1}{3\zeta(3)}$ and $B =  (1+\sqrt{3}) \frac{4\zeta(1/3)}{5\Gamma(2/3)^3\zeta(5/3)}$. 

We will not need the precise expression of these constants until section \ref{sec47} where we compute explicit expression of the constants. The important input from this theorem for us is that the order of $C_qC_{r^2} \le O( \prod_{e|q} e^{4/5} \prod_{e|r} e^{4/5}) \le O(qr)^{4/5+\epsilon}$. We will keep this fact in mind, but write $C_q$ and $C_{r^2}$ on the way. 

Another input we need is counting results from abelian extensions. Number of abelian extensions with local density are studied in \cite{M85, Wri89, Woo10a}. We will mainly use the estimate of abelian extension in the form of the uniformity estimates. Denote $N_q(A, X)$ to be the number of $A$-extensions $L$ over $\mathbb{Q}$ such that $q| \Disc(L)$, then we have the following estimate. 
 \begin{theorem}[\cite{JW17}, Theorem $4.13$]\label{uniab}
 	Let $A$ be a finite abelian group and $q$ be an integer, then
 	 $$ N_q(A,X) \le O(\frac{C^{\omega(q)}}{q^{1/a(A)}}) X^{1/a(A)}(\ln X)^{b(k, A)-1},$$
   where $C$ and the implied constant only depends on $k$ but not on $q$.
 \end{theorem}
 
 As for the notation, we denote the Dirichlet series $f(s)$ for $S_3$ cubic fields
$$f(s) =  \sum_{\Gal(K/\mathbb{Q}) = S_3} \frac{1}{\Disc(K)^s},$$ 
and denote $F_1(X) = \sum_{\Disc(K) \le X} 1$. Given a local condition $\Sigma$ which contains $\sigma_p$ at finitely many primes, we write $K\in\Sigma$ if $K_p\in \Sigma$ at those places. 
We denote
$$f_{\Sigma}(s) = \sum_{\Gal(K/\mathbb{Q}) = S_3, K\in \Lambda} \frac{1}{\Disc(K)^s},$$ and denote $F_{1, \Sigma} (X) = \sum_{\Disc(K)\le X, K \in \Sigma} 1$. Similarly for $A$ extensions, we will use $F_2(X)$, $F_{2, \Lambda}(X)$, $g(s)$ and $g_{\Lambda}(s)$. 
%
Then Theorem \ref{uniab} says that given a local condition $\Lambda$ on ramification behavior,
we have that 
	\begin{equation}\label{uniabs}
	\begin{aligned}
F_{2, \Lambda}(X) = O_{\epsilon}(D_{\Lambda})X^{1/a(A)+\epsilon},
	\end{aligned}
	\end{equation}
where $D_{\Lambda}$ can be bounded by the order of $O(\frac{C^{\omega(q)}}{|q|^{1/a(A)}})$ with $q$ associated with $\Lambda$. For brevity we will write $O$ instead of $O_{\epsilon}$ since $O_{\epsilon}$ only depends on $\epsilon$ and is independent of $\Lambda$. 

Notice that given a tuple $\varrho$ of local conditions as defined before, there will be naturally induced local condition on $K$ and $L$, called $\Sigma(\varrho)$ and $\Lambda(\varrho)$. The local conditions that come from the same $\varrho$ have the same support outside $T$ with non-trivial ramification restriction. Indeed, say $q_{ij}$ are the $i,j$-th square-free number in $\varrho$, then $\Sigma(\varrho)$ restricts the counting to $S_3$ cubic extensions that are partially ramified at $q_1 = \prod_{j}q_{ij}$ and are totally ramified at $q_2 = \prod_{j} q_{2j}$. Similarly for $A$ extensions we have $q_j$ for $1\le j\le |\mathcal{B}|$. In addition, at primes in $T$, $\Sigma(\rho)$ (and $\Lambda(\rho)$) also contains the corresponding $\sigma_p$ (and $\lambda_p$) in $w\in \mathcal{W}$. For technical reasons, if we do not include the conditions at $T$ into $\Sigma$, then we denote the smaller local condition $\Sigma'$. For brevity, we will call the corresponding coefficients depending on $\Sigma$ and $\Lambda$ in Theorem \ref{S3local} and \ref{uniab} by $A_{\Sigma}$, $B_{\Sigma}$, $C_{\Sigma}$ and $D_{\Lambda}$ in short. If we restrict the local condition to places outside $T$, we get the corresponding coefficient $A_{\Sigma'}$, $B_{\Sigma'}$, $C_{\Sigma'}$ and $D_{\Lambda'}$. Then if $\varrho = \rho \eta$, then $A_{\Sigma(\varrho)} = A_{\Sigma(\rho)} A_{\Sigma'(\eta)}$.

Recall that 
$$B(\varrho^1, X) = \sharp \{ (K, L) \mid (K, L) \in \varrho^{1},  \Disc(K)^m\Disc(L)^3 \le X \},$$
and $\varrho^{1}$ naturally gives a set of local specification $\Sigma(\varrho)$ for $K$ and $\Lambda(\varrho)$ for $L$, so equivalently
$$B(\varrho^1, X) = \sharp \{ (K, L) \mid K \in \Sigma(\varrho), L \in \Lambda(\varrho),  \Disc(K)\Disc(L)^{3/m} \le X^{1/m} \},$$
which is the product distribution of $F_{1, \Sigma}(X)$ and $F_{2, \Lambda}(X)$. 

Let's say $f_{\Sigma}(s) = \sum_{n} a_n\cdot n^{-s}$ and $g_{\Lambda}(s) = \sum_{k} b_k\cdot k^{-s}$. Then we have that
\begin{equation}
\begin{aligned}
B(\varrho^1, X) =& \sum_{\substack{K\in \Sigma, L\in \Lambda\\ \Disc(K)^m\Disc(L)^3 \le X}} 1 = \sum_{k^3\le X} b_kF_{1, \Sigma}(\frac{X^{1/m}}{k^{3/m}})\\
=& \sum_{k\le X^{1/3}} b_k \left(AA_{\Sigma} \frac{X^{1/m}}{k^{3/m}}+ B B_{\Sigma} (\frac{X^{1/m}}{k^{3/m}})^{5/6} + O(C_{\Sigma} (\frac{X^{1/m}}{k^{3/m}})^{2/3})\right)\\
=& AA_{\Sigma}X^{1/m} \sum_{k\le X^{1/3}}\frac{b_k}{k^{3/m}} + B B_{\Sigma} X^{5/6m} \sum_{k\le X^{1/3}}\frac{b_k}{k^{3/m\cdot 5/6}} \\
& + O(C_{\Sigma} X^{2/3m} \sum_{k\le X^{1/3}}\frac{b_k}{k^{3/m\cdot 2/3}})\\
 =& AA_{\Sigma}\cdot g_{\Lambda}(\frac{3}{m})X^{1/m}+ B B_{\Sigma}\cdot  g_{\Lambda}(\frac{5}{2m})X^{5/6m} + O(C_{\Sigma} \cdot g_{\Lambda}(\frac{2}{m}))X^{2/3m} \\
& + AA_{\Sigma}X^{1/m} \sum_{k\ge X^{1/3}}\frac{b_k}{k^{3/m}} + B B_{\Sigma} X^{5/6m} \sum_{k\ge X^{1/3}}\frac{b_k}{k^{3/m\cdot 5/6}} \\
& + O(C_{\Sigma}X^{2/3m} \sum_{k\ge X^{1/3}}\frac{b_k}{k^{3/m\cdot 2/3}}).\\
\end{aligned}
\end{equation}
Notice that in the last equality, we can take those values of $g_{\Lambda}(s)$ at $ s= 3/m, 5/2m, 2/m$ since the right most pole of $g_{\Lambda}(s)$ is at $s = \frac{1}{\ind(A)}$, which is smaller than $\frac{2}{m}$. Aside from the first two precise terms which will be the main term and the secondary term, we will denote the following errors $E_1$, $E_a$, $E_b$ and $E_c$ and analyze them one by one. 

\subsubsection{Bound on $E_i$ for $i = a,b,c$}
Let's first look at $E_a$. It suffices to bound the following weighted sum of $b_k$. By Abel summation, 
\begin{equation}
\begin{aligned}
E_a = AA_{\Sigma}X^{1/m} \sum_{k\ge X^{1/3}}\frac{b_k}{k^{3/m}} & = AA_{\Sigma} X^{1/m} \left(-\frac{F_{2, \Lambda}(X^{1/3})}{X^{1/m}} + \frac{3}{m}\int_{X^{1/3}}^{\infty} \frac{F_{2, \Lambda}(t)}{t^{3/m+1}} dt\right)\\
& = O(A_{\Sigma} D_{\Lambda})X^{1/3a(A)+ \epsilon}.
\end{aligned}
\end{equation}
Similarly, we get for $E_b$ that 
\begin{equation}
\begin{aligned}
E_b & = O(B_{\Sigma} D_{\Lambda})X^{1/3a(A)+ \epsilon},
\end{aligned}
\end{equation}
and for $E_c$ that 
\begin{equation}
\begin{aligned}
E_c & = O(C_{\Sigma} D_{\Lambda})X^{1/3a(A)+ \epsilon}.
\end{aligned}
\end{equation}

By Theorem \ref{S3local}, $A_{\Sigma}$ and $B_{\Sigma}$ are precise constants determined and are the local densities at $s=1$ and $s=5/6$, while $C_{\Sigma}$ is the  upper bound of the dependency for the error in the order of $(qr)^{4/5+\epsilon}$. So $E_c$ is the biggest one among $E_a$, $E_b$ and $E_c$, and we can combine them
\begin{equation}\label{E2}
\begin{aligned}
E_2 = E_a + E_b+E_c \le O(C_{\Sigma} D_{\Lambda})X^{1/3a(A)+ \epsilon}.
\end{aligned}
\end{equation}

\subsubsection{Bound on $E_1$}
To bound 
$$E_1 = O(C_{\Sigma}\cdot g_{\Lambda}(\frac{2}{m}))X^{2/3m},$$
it suffices to give a bound on $g_{\Lambda}(\frac{2}{m})$. From now on, we denote $b_j$ to be the generator of a tamely ramified inertia group, i.e. $\large \langle b_j  \large \rangle \in \Lambda_p  = \mathcal{B}$ for $p\notin T$. When we write $\ind(b_j)$, we mean the index of the group element.
\begin{lemma}\label{gvalue}
	Let $\Lambda$ be a local condition on ramification for $A$ extensions and let $g_{\Lambda}$ be the corresponding Dirichlet series, we have that at $s> 1/a(A)$, the value of $g_{\Lambda}(s)$ is bounded by
	$$g_{\Lambda} (s) \le  O(\prod_{j} \prod_{p|q_{j}} \frac{C}{p^{\ind(b_j)s}}) \le O(D_{\Lambda})^{a(A)s},$$
	where $q_j$ is the product of primes where the inertia group is $\large \langle b_j  \large \rangle\in \mathcal{B}$. The implied constant depends on $s$ but not on $\Lambda$. 
\end{lemma}
\begin{proof}
	Denote $J_{\mathbb{Q}}$ to be the id\`ele group of $\mathbb{Q}$. Notice that we can bound the number of A-extensions by the number of continous homomorphisms $ J_{\mathbb{Q}}/\mathbb{Q}^*\to A$, and it is equivalent to consider maps $\rho: \prod_p \mathbb{Z}_p^*\to A $ \cite{Wood14}. Local conditions on the abelian extensions could also be formulated by local conditions on $\rho$. At places outside $T$, the condition of $\Lambda$ is equivalent to the condition that the image of $\mathbb{Z}_p^{\times}$ under $\rho$ in $A$ is exactly specified as $\lambda_p\in \Lambda_p$. So we can also write $\rho\in\Lambda$ to specify the local condition on $\rho$.
	 We then have
	$$g_{\Lambda}(s) = \sum_{ K\in\Lambda} \frac{1}{\Disc(K)^s} \le \sum_{\rho : \rho\in \Lambda} \frac{1}{\Disc(\rho)^s} = \prod_p( \sum_{\rho_p: \mathbb{Z}_p^{\times}\to A , \rho_p\in\Lambda} \frac{1}{\Disc(\rho_p)^s}) = \tilde{g}_{\Lambda}(s).$$
	If $s> 1/a(A)$, then $g(s)$ and $\tilde{g}_{\Lambda}(s)$ are both convergent by \cite{M85, Wri89, Woo10a}. Also since $\tilde{g}$ and $\tilde{g}_{\Lambda}$ are both multiplicative, we can get the estimate for $\tilde{g}_{\Lambda}$ easily,
	$$\tilde{g}_{\Lambda}(s) = \tilde{g}(s) \cdot \frac{ \prod_{p}( \sum_{\rho_p: \mathbb{Z}_p^{\times}\to A , \rho_p\in\Lambda} \frac{1}{\Disc(\rho_p)^s}) }{ \prod_p( \sum_{\rho_p: \mathbb{Z}_p^{\times}\to A } \frac{1}{\Disc(\rho_p)^s}) } \le \tilde{g}(s)\cdot O(D_{\Lambda})^{a(A)s}.$$
\end{proof}
Plugging in the value from Lemma \ref{gvalue}, we get that 
\begin{equation}
\begin{aligned}
E_1 &= O(C_{\Sigma} \cdot D_{\Lambda}^{2a(A)/m})X^{2/3m}.
\end{aligned}
\end{equation}
Comparing with (\ref{E2}), we have that $$E = E_2+E_1 \le O(C_{\Sigma} \cdot D_{\Lambda}^{2a(A)/m})X^{2/3m}.$$ 

\subsection{Estimates of $B(\rho^0, XL_\rho)$ for Small $\rho$}\label{sec43}
In this section, we are going to compute the error for $B(\rho^0, X)$ which only involves small primes. Recall in (\ref{incexc}) that
$$
B(\rho^{0}, X) = \sum_{\varrho = \rho \eta} \mu(\eta) B(\varrho^{1}, X),
$$
where we define $\mu(\eta)$ to be $\prod_{i,j} \mu(k_{ij})$ with $k_{ij}$ is the $i,j$-th square-free integer in $\eta$. We expect the main terms from $B(\varrho^1, X)$ to contribute to the main term, so we will only look at the error terms. Denote $\Sigma(\rho)$ to be $\Sigma$ induced by $\rho$ and similarly for $\Lambda(\rho)$. Like we define $\Disc(\sigma_p)$, we could also define $\Disc(\Sigma(\rho))$ to be $\prod_p \Disc(\sigma_p)$ where the product is over all $p\in T$ and $p| \rho_{ij}$ for all $i$ and $j$, and define $\Disc(\Sigma'(\rho))$ to be $\prod_p \Disc(\sigma_p)$ where the product is over all $p| \rho_{ij}$ for all $i$ and $j$, then $\Disc(\Sigma(\rho \eta)) = \Disc(\Sigma(\rho)) \Disc(\Sigma'(\eta))$.

Notice that $B(\varrho^1, X) = 0$ when $\eta$ involves primes that are too large since 
\begin{equation}\label{Discbound}
\begin{aligned}
\Disc(\Sigma'(\eta))^m \Disc(\Lambda'(\eta))^3 = \prod_{i,j} k_{ij}^{m\ind(a_i) + 3\ind (b_j)} = k^{\beta}\le X^{*} = \frac{X}{\Disc(\Sigma(\rho))^m \Disc(\Lambda(\rho))^3}.
\end{aligned}
\end{equation}
For brevity we write $k = (k_{ij})$ as a vector and $\beta = (\beta_{ij}) = (m\ind(a_i) + 3\ind(b_j))$ as a vector of exponent. So there are two sources of error: one comes from the small $\eta$ where we apply the sieve; and the other one comes from the big $\eta$ where we pretend to have precise terms. 

For small $\eta$, by the inclusion-exclusion, the error is 
\begin{equation}
\begin{aligned}
W_1 & = \sum_{\substack{\varrho = \rho \eta\\ k^{\beta}< X^*}} \mu(\eta) O(C_{\Sigma(\varrho)}\cdot D_{\Lambda(\varrho)}^{2a(A)/m})X^{2/3m}\\
& \le O(C_{\Sigma(\rho)} \cdot D_{\Lambda(\rho)}^{2a(A)/m}) \cdot X^{2/3m} \sum_{ k^{\beta}<X^*} C_{\Sigma(\eta)} D_{\Lambda(\eta)}^{2a(A)/m}.
\end{aligned}
\end{equation}
The last inequality comes from the fact that $C_{\Sigma}$ and $D_{\Lambda}$ are multiplicative up to $O(1)$ at most. 

For big $\eta$, although $B(\varrho^1, X) = 0$, we would still like to use the main term and the secondary term in the same form. In order to compensate for that, we have the error coming from the main term
\begin{equation}
\begin{aligned}
W_2 & = \sum_{\substack{\varrho = \rho \eta\\ k^{\beta}> X^*}} O(A_{\Sigma(\varrho)}\cdot g_{\Lambda(\varrho)}(\frac{3}{m})X^{1/m}) \le O(A_{\Sigma(\rho)} \cdot D_{\Lambda(\rho)}^{3a(A)/m}) \cdot X^{1/m} \sum_{ k^{\beta}>X^*} A_{\Sigma(\eta)} D_{\Lambda(\eta)}^{3a(A)/m},\\
\end{aligned}
\end{equation}
and similarly for the secondary term,
\begin{equation}
\begin{aligned}
W_3 & = \sum_{\substack{\varrho = \rho \eta\\ k^{\beta}> X^*}} O(B_{\Sigma(\varrho)}\cdot  g_{\Lambda(\varrho)}(\frac{5}{2m}))X^{5/6m}\le O(B_{\Sigma(\rho)} \cdot D_{\Lambda(\rho)}^{5a(A)/2m}) \cdot X^{5/6m} \sum_{ k^{\beta}>X^*} B_{\Sigma(\eta)} D_{\Lambda(\eta)}^{5a(A)/2m}.
\end{aligned}
\end{equation}
\subsubsection{Bound on $W_1$}
We look into the following sum
$$ R_1 = \sum_{k^{\beta}<X^*} C_{\Sigma(\eta)} D_{\Lambda(\eta)}^{2a(A)/m}.$$
To be more precise, recall that $k_{ij}$ is the $i,j$-th square-free number for $\eta$, then 
%
\begin{equation}\label{Cexpl}
\begin{aligned}
C_{\Sigma(\eta)} =O( \prod_{j} k_{1j}^a \prod_{j} k_{2j}^b),
\end{aligned}
\end{equation}
where $a$ and $b$ are such numbers that $C_p = p^a$ and $C_{p^2} = p^b$ in Theorem \ref{S3local}. We know from Theorem \ref{S3local} that we can take $a = b = 4/5$. We will keep $a$ and $b$ to see how much we need from them. For $\Lambda$, 
\begin{equation}\label{Dexpl}
\begin{aligned}
D_{\Lambda(\eta)}^{2a(A)/m} = O\left(\prod_{j} \prod_{i} \prod_{p|k_{ij}} C' p^{-2\ind(b_j)/m}\right) \le O_{\epsilon}\left(\prod_{j} (\prod_{i} k_{ij})^{-2\ind(b_j)/m +\epsilon}\right),
\end{aligned}
\end{equation}
where $C' = C^{-2a(A)/m}$ is a new absolute constant depending only on $A$. So the sum $R_1$ could be bounded by a sum of multi-variable polynomial over a bounded region, 
\begin{equation}
\begin{aligned}
R_1 \le O(\sum_{k^{\beta} \le X^{*}} k^{\gamma}).
\end{aligned}
\end{equation}
Here $\beta$ and $\gamma$ could be determined by (\ref{Cexpl}), (\ref{Dexpl}) and (\ref{Discbound}). The summation is considered in the following elementary calculus result. It can be proved by direct computation. 
\begin{lemma}\label{calcweight}
	Given a vector of component $\beta$ and $\gamma$ such that $\beta_{i}>0$ for all $1\le i\le n$, if there exists $i$ such that $\gamma_i \ge -1$, then the following summation is bounded
	\begin{equation}
	\begin{aligned}
	\sum_{k^{\beta} \le X} k^{\gamma} \le O_{\epsilon}(X^{a(\beta, \gamma)+\epsilon}),
	\end{aligned}
	\end{equation}
	where $a(\beta, \gamma) = \max_{1\le  i\le n} \{ \frac{\gamma_i+1}{\beta_i} \}$. If $\gamma_i < -1$ for all $i$, then the sum is bounded by $O(1)$. 
\end{lemma}
In our case, $\beta$ and $\gamma$ are indexed by $i$ and $j$. The quotient is computed to be
$$\frac{\gamma_{ij} +1}{\beta_{ij}} = \frac{a - 2\ind(b_j)/m +1}{m\ind(a_i)+3\ind(b_j)},$$
for $i=1$, and similarly for $i = 2$
$$\frac{\gamma_{ij} +1}{\beta_{ij}} = \frac{b - 2\ind(b_j)/m +1}{m\ind(a_i)+3\ind(b_j)},$$ after plug in (\ref{Cexpl}), (\ref{Dexpl}) and (\ref{Discbound}). Observe that if the numerator is positive, then this quantity is largest when $\ind(b_j) = \ind(A)$, i.e., 
$$a(\beta, \gamma) m = \max \left\{\frac{a- 2\ind(A)/m +1}{1+ 3\ind(A)/m},  \frac{b- 2\ind(A)/m +1}{2+ 3\ind(A)/m} \right  \}.$$ Since we have in Theorem \ref{S3local} that $C_p = p^a = C_{q^2}$ for $a = b= 4/5$, in our situations, the quantity is also largest when $\ind(a_i) = \ind(A)$, i.e., 
$$a(\beta, \gamma) m =\frac{a- 2\ind(A)/m +1}{1+ 3\ind(A)/m}.$$
%
It is possible that the above expression is negative for some $A$. In that case, the summation $R_1$ is $O(1)$, so we define $a(\beta, \gamma) = 0$ for such $A$. 

Plugging in $R_1$, we get $W_1$ for $B(\rho^0, X)$ that
\begin{equation}
\begin{aligned}
W_1 & \le O(C_{\Sigma(\rho)} \cdot D_{\Lambda(\rho)}^{2a(A)/m}) \cdot X^{2/3m} (X^*)^{ a(\beta, \gamma) +\epsilon}.
\end{aligned}
\end{equation}

\subsubsection{Bound on $W_2$ and $W_3$}
In this subsection, we look into $W_2$ and $W_3$ in a similar way, and we will show that they are small. Therefore only the error from small $\eta$ makes main contribution to the error of $B(\rho^0, X)$. 

Denote 
$$R_2 = \sum_{ k^{\beta}>X^*} A_{\Sigma(\eta)} D_{\Lambda(\eta)}^{3a(A)/m}.$$

We will need a similar lemma to deal with $R_2$. 
\begin{lemma}\label{calcweight2}
	Given a vector of component $\beta$ and $\gamma$ such that $\beta_{i}>0$ for all $1\le i\le n$, if $\gamma_i < -1$ for all $i$, then the following summation is bounded
	\begin{equation}
	\begin{aligned}
	\sum_{k^{\beta} \ge X} k^{\gamma} \le O_{\epsilon}(X^{a(\beta, \gamma)+\epsilon}),
	\end{aligned}
	\end{equation}
	where $a(\beta, \gamma) = \max_{1\le  i\le n} \{ \frac{\gamma_i+1}{\beta_i} \}$.
\end{lemma}
The exponent $\beta'_{ij}$ is the same as $\beta_{ij}$ in $R_1$, but the exponent $\gamma'_{ij}$ is different. By description of $A_{\Sigma}$ and $D_{\Lambda}$, the quotient is
$$\frac{\gamma'_{ij} +1}{\beta'_{ij}} = \frac{-\ind(a_i) - 3\ind(b_j)/m +1}{m\ind(a_i)+3\ind(b_j)} = \frac{1}{m}(-1+\frac{1}{\ind(a_i) + 3\ind(b_j)/m})\le \frac{-3\ind(A)/m}{m + 3\ind(A)},$$
where in the last inequality we take $\ind(a_i) = \ind(S_3)$ and $\ind(b_j) = \ind(A)$.
Therefore $$a(\beta', \gamma') m = \frac{-3\ind(A)/m}{ 1+ 3\ind(A)/m}.$$
By Lemma \ref{calcweight2} and description of $A_{\Sigma(\rho)}$, $D_{\Lambda(\rho)}$, we have
\begin{equation}
\begin{aligned}
W_2 & = O(A_{\Sigma(\rho)} \cdot D_{\Lambda(\rho)}^{3a(A)/m}) \cdot X^{1/m} \cdot (X^*)^{a(\beta', \gamma') +\epsilon} \le O(X^*)^{1/m + a(\beta', \gamma') + \epsilon}.
\end{aligned}
\end{equation}
Similarly for $W_3$, the exponent
$$a(\beta'', \gamma'') m= \frac{-5\ind(A)/2m}{ 1+ 3\ind(A)/m} =\frac{5}{6} \cdot a(\beta', \gamma')m,$$
and
\begin{equation}
\begin{aligned}
W_3 & = O(B_{\Sigma(\rho)} \cdot D_{\Lambda(\rho)}^{5a(A)/2m}) \cdot X^{5/6m}  (X^*)^{a(\beta'', \gamma'')+\epsilon} \le O( \Disc(\Sigma(\rho))^{-1/6} (X^*)^{5/6m + a(\beta'', \gamma'') + \epsilon}).
\end{aligned}
\end{equation}
Therefore the bound on $W_3$ is smaller than that of $W_2$, so it suffices to compare that of $W_2$ with $W_1$. Notice that 
$$-\frac{1}{3} + a(\beta, \gamma)m - a(\beta', \gamma')m \ge \frac{a+ 2/3}{1+ 3\ind(A)/m}>0 ,$$
we have $$W_2 \le O( C_{\Sigma(\rho)} \Disc(\Sigma(\rho))^{2/3}  )( X^{*})^{2/3m + a(\beta, \gamma) +\epsilon} = O(C_{\Sigma(\rho)} \cdot D_{\Lambda(\rho)}^{2a(A)/m}) \cdot X^{2/3m} (X^*)^{ a(\beta, \gamma) +\epsilon},$$
so 
$$W_1 + W_2 + W_3 \le O(C_{\Sigma(\rho)} \cdot D_{\Lambda(\rho)}^{2a(A)/m}) \cdot X^{2/3m} (X^*)^{ a(\beta, \gamma) +\epsilon}.$$ 

\subsubsection{Error for $B(\rho^0, XL_\rho)$}
Finally, we replace $X$ with $XL_{\rho}$ in $B(\rho^{0}, XL_{\rho})$ and denote the error by $E_{\rho}$. Plugging in
$$X^{*} =  \frac{XL_{\rho}}{\Disc(\Sigma(\rho))^m \Disc(\Lambda(\rho))^3},$$
we get
\begin{equation}
\begin{aligned}
E_{\rho} & \le O(C_{\Sigma(\rho)} \cdot D_{\Lambda(\rho)}^{2a(A)/m}) \cdot (XL_{\rho})^{2/3m} O_{\epsilon}( \frac{XL_{\rho}}{\Disc(\Sigma(\rho))^m \Disc(\Lambda(\rho))^3})^{a(\beta, \gamma)+\epsilon}\\
& \le O_{\epsilon}(X^{2/3m + a(\beta, \gamma)+\epsilon} ) O(C_{\Sigma(\rho)}  D_{\Lambda(\rho)}^{2a(A)/m}  L_{\rho}^{2/3m + a(\beta, \gamma) + \epsilon})  (\Disc(\Sigma(\rho))^m \Disc(\Lambda(\rho))^3)^{-a(\beta, \gamma) -\epsilon}\\
& \le O_{\epsilon}(X^{2/3m + a(\beta, \gamma)+\epsilon} ) \prod_{i,j} k_{ij}^{e(i,j)},
\end{aligned}
\end{equation}
where $$e(1,j) = a+ 2/3 - \ind((12)(3), b_j)\cdot (2/3m + a(\beta, \gamma)) + \epsilon,$$ and $$e(2, j) = b+ 4/3 - \ind((123), b_j)\cdot (2/3m + a(\beta, \gamma))+ \epsilon.$$
Here $\ind((12)(3), bj)$ means the index of the group element $((12)(3), b_j) \in S_3\times A$. 
\subsection{Estimates of $B(\rho^0, XL_{\rho})$ for Large $\rho$}\label{sec44}
In this section, we will use uniformity estimates to determine the tail estimate for $B(\rho^{0}, XL_{\rho})$ for $\rho$. The expression will hold uniformly for all $\rho$, but it will be especially helpful when $\rho$ involves relatively larger prime numbers. 

Recall in Theorem \ref{unipato} and \ref{uniab}, we get uniformity estimates for $S_3$ cubic extension and $A$-extension with restriction on ramification, which states that 
$$ \sharp \{ K \mid \Gal(K) = S_3, K\in \Sigma(\rho), \Disc(K )\le X  \}  = O (\frac{X}{\prod_j k_{1j}^{1/6-\epsilon} \prod_jk_{2j}^{2-\epsilon}}),$$
and 
$$ \sharp \{ L \mid \Gal(L) = A, L\in \Lambda(\rho), \Disc(K )\le X  \}  = O (\frac{X^{1/a(A) +\epsilon}}{\prod_j (k_{1j}k_{2j})^{\ind(b_j)/a(A)-\epsilon}}).$$

Since $\rho^0$ requires more restriction on places outside $k_{ij}$, we have that
$$B(\rho^0, XL_{\rho}) \le B(\rho^1, XL_{\rho}).$$
Applying Theorem \ref{proneq} on $\Disc_{res}(K) = \frac{\Disc(K)}{\Disc(\Sigma(\rho))}$ and $\Disc_{res}(L) = \frac{\Disc(L)}{\Disc(\Lambda(\rho))}$, we get

\begin{equation}
\begin{aligned}
& B(\rho^1, XL_{\rho}) \\
= & \sharp \{ (K, L )\mid K\in \Sigma(\rho), L\in \Lambda(\rho), \Disc_{res}(K)^m \Disc_{res}(L)^3 \le \frac{XL_\rho}{ \Disc(\Sigma(\rho))^m \Disc(\Lambda(\rho))^3}   \}\\
= & \sharp \{ (K, L )\mid K\in \Sigma(\rho), L\in \Lambda(\rho), \Disc_{res}(K)^m \Disc_{res}(L)^3 \le \frac{X}{ \Disc(\Sigma(\rho), \Lambda(\rho))} \}\\
\le & O(\prod_j k_{1j}^{5/6+\epsilon} \prod_jk_{2j}^{\epsilon}) (\frac{X}{ \Disc(\Sigma(\rho), \Lambda(\rho))})^{1/m} \\
= & O(X^{1/m}) \prod_{i,j} k_{ij}^{d(i,j)},
\end{aligned}
\end{equation}
where $$d(1,j) = \frac{5}{6} + \epsilon -\ind((12)(3), b_j)/m,$$ and $$d(2,j) = \epsilon - \ind((123), b_j)/m.$$ These tail estimates will all be error terms, and we will denote it by
$$D_{\rho} = B(\rho^0, XL_{\rho}) \le B(\rho^1, XL_{\rho}) \le O(X^{1/m}) \prod_{i,j} k_{ij}^{d(i,j)}.$$
\subsection{Optimization}\label{sec45}
In this section, we will combine the error estimates in previous sections, and balance between errors in the small range and the large range to optimize the  error overall.

Recall that in the small range we get the error 
$$E_{\rho} = O(X^{2/3m + a(\beta, \gamma)+\epsilon} ) \prod_{i,j} k_{ij}^{e(i,j)},$$
and in the large range we get the error
$$D_{\rho} = O(X^{1/m}) \prod_{i,j} k_{ij}^{d(i,j)}.$$
So to take advantage of both estimate, we will use the sieve argument when $$E_{\rho} \le D_{\rho},$$
which is equivalent to
$$\prod_{i,j} k_{ij}^{\delta(i,j)} \le X^{1/3m -a(\beta, \gamma) -\epsilon} = Q,$$
where $\delta(i,j) = e(i,j)-d(i,j)$.


So the error overall will be
\begin{equation}
\begin{aligned}
E & = \sum_{\substack{\rho\\ \prod_{i,j} k_{i,j}^{\delta(i,j)} \le Q}} E_{\rho} + \sum_{\substack{\rho\\ \prod_{i,j} k_{i,j}^{\delta(i,j)} \ge Q}} D_{\rho}\\
& = E_S + E_L
\end{aligned}
\end{equation}

\begin{enumerate}
	\item  \textbf{Estimates for $E_S$}\\
The sum $E_S$ for the small range is
\begin{equation}
\begin{aligned}
E_S & =O_{\epsilon}(X^{2/3m + a(\beta, \gamma)+\epsilon} )\sum_{\substack{\rho\\ \prod_{i,j} k_{ij}^{\delta(i,j)} \le Q}} 
\prod_{i,j} k_{ij}^{e(i,j)}\\
& = O_{\epsilon}(X^{2/3m + a(\beta, \gamma)+ \epsilon})\cdot Q^{ \max \{  \frac{e(i,j)+1}{\delta(i,j)}     \}}.
\end{aligned}
\end{equation}
For the second equality, we apply Lemma \ref{calcweight} since there exists $e(i,j)> -1$ and for all $i$ and $j$, $\delta(i,j) > 0$.

\item  \textbf{Estimates for $E_L$}\\
The sum $E_L$ for the large range is
\begin{equation}
\begin{aligned}
E_L & =O(X^{1/m})  \sum_{\substack{\rho\\ \prod_{i,j} k_{ij}^{\delta(i,j)} \ge Q}} 
\prod_{i,j} k_{ij}^{d(i,j)}\\
& = O(X^{1/m})\cdot Q^{ \max \{  \frac{d(i,j)+1}{\delta(i,j)} \}}.
\end{aligned}
\end{equation}
For the second equality, we apply Lemma \ref{calcweight2} since $d(i,j)< -1$ and $\delta(i,j) > 0$ for all $i$ and $j$.
\end{enumerate}

To sum up, for the small range, we use estimates with precise first term, secondary term and an error term $E_{\rho}$; for large range, we use estimates which is purely error term $D_{\rho}$. Since the first term and secondary term are both small comparing to $D_{\rho}$,
$$O(B_{\Sigma}\cdot  g_{\Lambda}(\frac{5}{2m}))(XL_{\rho})^{5/6m}\le O(A_{\Sigma}\cdot g_{\Lambda}(\frac{3}{m}))(XL_{\rho})^{1/m} \le O(X^{1/m}) \prod_{i,j} k_{ij}^{d(i,j)},$$
by comparing the exponent for each $k_{ij}$, we could pretend that we use estimates with a precise main term and a secondary term with the error $D_{\rho}$ without harm. Finally we get that the error is
\begin{equation}
\begin{aligned}
E & =  O_{\epsilon}(X^{2/3m + a(\beta, \gamma)+ \epsilon})\cdot Q^{ \max \{  \frac{e(i,j)+1}{\delta(i,j)}\}} + O(X^{1/m})\cdot Q^{ \max \{  \frac{d(i,j)+1}{\delta(i,j)} \}},
\end{aligned}
\end{equation}
where $Q = X^{1/3m -a(\beta, \gamma) -\epsilon}$, and $a(\beta, \gamma)$, $e(i,j)$, $d(i,j)$ and $\delta(i,j)$ are constants depending on $A$. 

Therefore finally it reduces to the question if we could show for $A$ that 
\begin{equation}\label{inposas}
\begin{aligned}
\frac{2}{3m} + a(\beta, \gamma) + \epsilon + (\frac{1}{3m}-a(\beta, \gamma)-\epsilon)\cdot \max_{i,j} \{ \frac{e(i,j) +1}{\delta(i,j)}\} < \frac{1}{m},
\end{aligned}
\end{equation}
and 
\begin{equation}\label{inposal}
\begin{aligned}
\frac{1}{m}+ (\frac{1}{3m}-a(\beta, \gamma)-\epsilon)\cdot \max_{i,j} \{ \frac{d(i,j) +1}{\delta(i,j)}\} < \frac{1}{m}.
\end{aligned}
\end{equation}
If we could show the above inequalities for $A$, then we succeed in proving a power saving error for $N(S_3\times A, X)$. Moreover, if we could show the two inequalities with the right hand side replaced by $\frac{5}{6m}$,
\begin{equation}\label{inseds}
\begin{aligned}
\frac{2}{3m} + a(\beta, \gamma) +\epsilon + (\frac{1}{3m}-a(\beta, \gamma)-\epsilon)\cdot \max_{i,j} \{ \frac{e(i,j) +1}{\delta(i,j)}\} < \frac{5}{6m},
\end{aligned}
\end{equation}
and 
\begin{equation}\label{insedl}
\begin{aligned}
\frac{1}{m}+ (\frac{1}{3m}-a(\beta, \gamma)-\epsilon)\cdot \max_{i,j} \{ \frac{d(i,j) +1}{\delta(i,j)}\} < \frac{5}{6m},
\end{aligned}
\end{equation} then we will succeed in saving the secondary term in the order of $X^{5/6m}$. Since the inequality are all strict, we could totally ignore those $\epsilon$ since they could be arbitrarily small. 



\subsection{Proof of the Main Theorem}\label{sec46}
In this section, we will prove the main theorem by verifying (\ref{inposas}),(\ref{inposal}), (\ref{inseds}) and (\ref{inseds}). To do that, we will compute explicitly the quantities of these parameters of an abelian group $A$: $a(\beta, \gamma)$, $e(i,j)$, $d(i,j)$ and $\delta(i,j)$. In the following discussion, let us denote an important quantity associated to $A$ by
$$\Delta = \frac{\ind(A)}{m} = \frac{p-1}{p},$$
in which $p$ is the smallest prime divisor of $m$. 

Firstly, recall that if the following quantity is positive then
$$a(\beta, \gamma) = \frac{a- 2\ind(A)/m +1}{m+ 3\ind(A)},$$
otherwise, 
$$a(\beta, \gamma) = 0.$$
By solving for $a(\beta, \gamma) = 0$ and plugging in $a = 4/5$, we get that if $p> 7$, then $a(\beta, \gamma) = 0$. On the other hand, if $p = 3, 5, 7$, then
$$a(\beta, \gamma) m = \frac{a-2\Delta +1}{1+ 3\Delta}.$$
Secondly, recall that we have
 $$e(1,j) = a+ 2/3 - \ind((12)(3), b_j)\cdot (2/3m + a(\beta, \gamma)) + \epsilon,$$ $$e(2, j) = b+ 4/3 - \ind((123), b_j)\cdot (2/3m + a(\beta, \gamma))+ \epsilon,$$
 $$d(1,j) = \frac{5}{6} + \epsilon -\ind((12)(3), b_j)/m,$$$$d(2,j) = \epsilon - \ind((123), b_j)/m,$$
 where $a = b = 4/5$ as in Theorem \ref{S3local}. 
 We will need the following lemma.
 \begin{lemma}[\cite{JW17}, Lemma $2.4$]\label{delta3}
 	Let $A$ be an abelian group of odd order $m$ and $(12)$, $(123)$ be elements in $S_3$. Then for all $c\in A$, $\ind((12),c)/m > 2$, $\ind((123),c)/m>1$.
 \end{lemma}
 
 \begin{proof}[\textbf{Proof of Theorem \ref{Thm1} and Theorem \ref{Thm2}}]
 	
 It suffices to prove the inequality (\ref{inseds}) and (\ref{insedl}) for $p>5$ and (\ref{inposas}) and (\ref{inposal}) for $p = 3, 5$. The key quantity is the maximum of $\frac{e(i,j) +1}{\delta(i,j)}$ and $\frac{d(i,j) +1 }{\delta(i,j)}$ over $j$ for $i = 1, 2$. 
 We will call them $U_i$ and $V_i$ for $i = 1, 2$ correspondingly. Notice that for each fixed $i,j$
 $$
 \frac{2}{3m} + a(\beta, \gamma) + (\frac{1}{3m}-a(\beta, \gamma))\cdot  \frac{e(i,j) +1}{\delta(i,j)} = \frac{1}{m}+ (\frac{1}{3m}-a(\beta, \gamma))\cdot \frac{d(i,j) +1}{\delta(i,j)}, 
 $$
 and in all of our cases, we can check that the maximum value of $U_i$ and $V_i$ are obtained when $\ind(b_j) = \ind(A)$. Therefore to check (\ref{inseds}) is equivalent to check (\ref{insedl}) and similarly for (\ref{inposas}) and (\ref{inposal}). It suffices to check for $U_i$. 
 
 When $p>7$, by Theorem \ref{indge} we have 
 $$\frac{e(1,j)+1}{\delta(1,j)} = \frac{1+ a - 4\ind(b_j)/3m+\epsilon}{a+1/6+ 2\ind(b_j)/3m} \le \frac{1+a - 4\Delta/3 + \epsilon}{1/6 +a +2\Delta/3} = U_1,$$
 where the maximum is taken when $\ind(b_j) = \ind(A)$. Similarly,
  $$\frac{e(2,j)+1}{\delta(2,j)} \le \frac{1+b - 2\Delta/3+\epsilon}{2+b+\Delta/3} = U_2.$$
%
So the inequality (\ref{inseds}) becomes purely dependent on $\Delta$: 
\begin{equation}\label{ingt7}
\begin{aligned}
\frac{2}{3} +\epsilon + (\frac{1}{3}-\epsilon)\cdot U_i  & < \frac{5}{6},
\end{aligned}
\end{equation} 
for $i = 1,2$. We can check this holds for $p> 7$. 

When $p = 5, 7$, we have 
$$U_1 = \frac{1+a - 4\Delta/3 - (a-2\Delta +1)(1+2\Delta)(1+3\Delta)^{-1} +\epsilon}{1/6 +a +2\Delta/3 -(a-2\Delta +1)(1+2\Delta)(1+3\Delta)^{-1}},$$
$$U_2 =\frac{1+b - 2\Delta/3 - (a-2\Delta +1)(2+\Delta)(1+3\Delta)^{-1}  + \epsilon}{2+b+\Delta/3 - (a-2\Delta +1)(2+\Delta)(1+3\Delta)^{-1}  }.$$
It suffices to check the following holds
\begin{equation}\label{in57}
\begin{aligned}
\frac{2}{3} + \frac{a-2\Delta +1}{1+ 3\Delta} +\epsilon + (\frac{1}{3}-\frac{a-2\Delta +1}{1+ 3\Delta}-\epsilon)\cdot U_i  & < \frac{5}{6},
\end{aligned}
\end{equation} 
for $p = 7$, and when $p = 5$ the inequality holds when the right hand side is $1$, which finishes the proof of Theorem\ref{Thm1} and Theorem \ref{Thm2} for $A$ with $p>3$. 

When $ p = 3$, we just need to compute more carefully. Now $$a(\beta, \gamma) m =  \frac{a-1/3}{3} = \frac{7}{45},$$
and the inequality for $U_1$ remains the same since $(12)(3)$ has order $2$, which is relatively prime to order of $g$ for any $g\in A$. So it suffices to check (\ref{in57}) holds for $U_1$ when $\Delta = 2/3$ and the right hand side is $1$. For $U_2$, plugging $a(\beta, \gamma)m = 7/45$ into (\ref{in57}) and rearranging the terms, it suffices to check that
\begin{equation}\label{in3}
\begin{aligned}
 \frac{e(2,j)+1}{\delta(2,j)} < 1,\\
 \frac{d(2,j)+1}{\delta(2,j)} <0,
\end{aligned}
\end{equation} 
which is equivalent to
$$d(2,j) = \epsilon -\ind ((123), b_j)/m < -1.$$
It follows from the Lemma \ref{delta3}.
\end{proof}

\subsection{Constants for the Main Term and the Secondary Term}\label{sec47}
In this section, we are going to compute the precise constants for the main term and the secondary term when $A$ is a cyclic group with prime order $m = l$ for $l > 5$. 

We will first consider all continuous homomorphisms $G_{\mathbb{Q}}\to C_l$ instead of $C_l$ extensions of $\mathbb{Q}$ for simplicity of computation of main term. The two quantity differ by a trivial map up to an action of $\text{Aut}(C_l)$. The generating series for such maps is
$$g(s) = \left(1 + (l+1)l^{-2(l-1)s}\right)\prod_{p\ne l, p\equiv 1 \text{ mod }l} \left(1 + (l-1) p^{-(l-1)s}\right). $$

Recall that the precise terms for $B(\varrho^1, XL_{\rho})$ is
$$AA_{\Sigma(\varrho)}\cdot g_{\Lambda(\varrho)}(\frac{3}{m})(XL_{\rho})^{1/m} + B B_{\Sigma(\varrho)}\cdot  g_{\Lambda(\varrho)}(\frac{5}{2m})(XL_{\rho})^{5/6m},$$
so the main term for $B(\rho^0, XL_{\rho})$ is
$$
A(XL_{\rho})^{1/m} A_{\Sigma(\rho)}\cdot  g_{\Lambda(\rho)}(\frac{3}{m})  \sum_{\eta}\mu(\eta) A_{\Sigma'(\eta)}\cdot g_{\Lambda'(\eta)}(\frac{3}{m}),$$
and the secondary term for $B(\rho^0, XL_{\rho})$ is
$$ B(XL_{\rho})^{5/6m}B_{\Sigma(\rho)}\cdot  g_{\Lambda(\rho)}(\frac{5}{2m}) \sum_{\eta} \mu(\eta)B_{\Sigma'(\eta)}\cdot  g_{\Lambda'(\eta)}(\frac{5}{2m}),
$$
where both sums are over all $\eta$ that is relatively prime to $\rho$. Finally we sum over all $\rho$ and get the main term for the whole counting 
$$ AX^{1/m}  \sum_{\rho} L_{\rho}^{1/m} A_{\Sigma(\rho)}\cdot g_{\Lambda(\rho)}(\frac{3}{m})  \sum_{\eta}\mu(\eta) A_{\Sigma'(\eta)}\cdot g_{\Lambda'(\eta)}(\frac{3}{m}),$$
with a secondary term
$$ BX^{5/6m} \sum_{\rho} L_{\rho}^{5/6m} B_{\Sigma(\rho)}\cdot  g_{\Lambda(\rho)}(\frac{5}{2m}) \sum_{\eta} \mu(\eta)B_{\Sigma'(\eta)}\cdot  g_{\Lambda'(\eta)}(\frac{5}{2m}),$$
where we sum over all possible $\rho$. 

In this specific case, an extension $L$ with the Galois group $C_l$ could only be wildly ramified at $l$. At $l>3$, an $S_3$ cubic extension $K$ could be tamely ramified so $l$ is the only place we need be careful about wildly ramification. On the other hand $C_l$ is cyclic with prime order, so there is only one type of tamely ramification, and $S_3$ has two types of tamely ramification, so the tamely ramification part in a local condition $\rho$, could be parametrized by a pair of relatively prime square-free integers, say $q$ and $r$. Similarly for $\eta$, say $k$ and $l$, with $klpq$ square-free. So plugging in all the constant we get the coefficient for the main term:

\begin{equation}\label{mainterm}
\begin{aligned}
\mathcal{C}_1 = &A \sum_{\rho} L_{\rho}^{1/m} A_{\Sigma(\rho)}\cdot g_{\Lambda(\rho)}(\frac{3}{m})  \sum_{\eta}\mu(\eta) A_{\Sigma'(\eta)}\cdot g_{\Lambda'(\eta)}(\frac{3}{m})\\
= & \frac{1}{3\zeta(3)} \cdot c_l\cdot  \sum_{\substack{q,r,k,l \\ \forall p|qrkl, p\equiv 1 \text{ mod } l}}  \mu(k) \mu(l)\frac{1}{q^{\Delta} r^{2\Delta}} \cdot \prod_{p|qk} \frac{C_p^{-1}}{p} \cdot \prod_{p|rl}\frac{C_p^{-1} }{p^2} \cdot g_(\frac{3}{l}) \prod_{p|qrkl} \frac{(l-1)p^{-3\Delta}}{1+ (l-1)p^{-3\Delta}}\\
= & \frac{g(3/l)}{3\zeta(3)} \cdot c_l\cdot\prod_{p \equiv 1 \text{ mod }l } \left\{1 + p^{-\Delta}\cdot \frac{C_p^{-1}}{p} \cdot \frac{(l-1)p^{-3\Delta}}{1+(l-1)p^{-3\Delta}}+ p^{-2\Delta}\cdot \frac{C_p^{-1}}{p^2} \cdot \frac{(l-1)p^{-3\Delta}}{1+(l-1)p^{-3\Delta}} \right.\\
&\left.- \frac{C_p^{-1}}{p} \cdot \frac{(l-1)p^{-3\Delta}}{1+(l-1)p^{-3\Delta}} - \frac{C_p^{-1}}{p^2} \cdot \frac{(l-1)p^{-3\Delta}}{1+(l-1)p^{-3\Delta}} \right\},
\end{aligned}
\end{equation} 
where $C_p = 1+ p^{-1}+ p^{-2}$ is the normalizing factor for the local density at $s = 1$ for $S_3$ extensions, and $$g(3/l) = (1 + (l+1)l^{-6\Delta})\prod_{p\ne l, p\equiv 1 \text{ mod }l} (1 + (l-1) p^{-3\Delta}),$$
and the local factor at $l$ 
$$c_l = \sum_{(\sigma_l,\lambda_l)} \frac{\Disc(\sigma_l)\Disc(\lambda_l)^{3/l}}{\Disc(\sigma_l, \lambda_l)^{1/l}} \cdot \frac{g_{\lambda_l}(3/l)}{g(3/l)} \cdot A_{\sigma_l}.$$
Similarly, we can compute the constant for the secondary term
\begin{equation}\label{sedterm}
\begin{aligned}
\mathcal{C}_2 = &B \sum_{\rho} L_{\rho}^{5/6m} B_{\Sigma(\rho)}\cdot  g_{\Lambda(\rho)}(\frac{5}{2m}) \sum_{\eta} \mu(\eta)B_{\Sigma'(\eta)}\cdot  g_{\Lambda'(\eta)}(\frac{5}{2m})\\
= &B \cdot g_(\frac{5}{2l}) \cdot d_l\cdot \sum_{\substack{q,r,k,l \\ \forall p|qrkl\\ p\equiv 1 \text{ mod } l}} \left\{  \mu(k) \mu(l)(\frac{1}{q^{\Delta} r^{2\Delta}})^{5/6} \cdot \prod_{p|qk} \frac{(1+ p^{-1/3})^2}{K_p \cdot p} \right.\\ &\quad\quad\quad\quad\quad\quad\quad\quad\quad\quad\quad\quad \left.\cdot \prod_{p|rl}\frac{(1+ p^{-1/3}) }{K_p\cdot p^2} \cdot\prod_{p|qrkl} \frac{(l-1)p^{-5\Delta/2}}{1+ (l-1)p^{-5\Delta/2}}\right\}\\
= & B\cdot g(5/2l) \cdot d_l\cdot \prod_{p \equiv 1 \text{ mod }l } \left\{1 + p^{-5\Delta/6}\cdot \frac{(1+ p^{-1/3})^2}{K_p\cdot p} \cdot \frac{(l-1)p^{-5\Delta/2}}{1+(l-1)p^{-5\Delta/2}}\right. \\ & + p^{-5\Delta/3}\cdot \frac{1+ p^{-1/3}}{K_p\cdot p^2} \cdot \frac{(l-1)p^{-5\Delta/2}}{1+(l-1)p^{-5\Delta/2}}\\
& \left. - \frac{(1+ p^{-1/3})^2}{K_p\cdot p} \cdot \frac{(l-1)p^{-5\Delta/2}}{1+(l-1)p^{-5\Delta/2}} - \frac{1+ p^{-1/3}}{K_p\cdot p^2} \cdot \frac{(l-1)p^{-5\Delta/2}}{1+(l-1)p^{-5\Delta/2}}\right\},
\end{aligned}
\end{equation} 
where $$B =  (1+\sqrt{3}) \frac{4\zeta(1/3)}{5\Gamma(2/3)^3\zeta(5/3)},$$ is the constant for the secondary term for $S_3$ extensions, and $$K_p = \frac{(1-p^{-5/3})(1+p^{-1})}{1-p^{-1/3}},$$
is the normalizing factor for the local density at $s = 5/6$ for $S_3$ extensions, and
$$d_l = \sum_{(\sigma_l,\lambda_l)} \frac{\Disc(\sigma_l)^{5/6}\Disc(\lambda_l)^{5/2l}}{\Disc(\sigma_l, \lambda_l)^{5/6l}} \cdot \frac{g_{\lambda_l}(5/2l)}{g(5/2l)} \cdot B_{\sigma_l}.$$

%

Notice that we are counting continous homomorphisms from $G_{\mathbb{Q}}$ to $S_3\times C_l$ which are surjective onto the $S_3$ component up to an action of $\text{Aut}(C_l)$ on the $C_l$ component, therefore the true value for the constant of the main term is
$$C_1 = \frac{1}{l-1}\cdot (\mathcal{C}_1 - A),$$
and the value for the constant of the secondary term is
$$C_2 = \frac{1}{l-1}\cdot (\mathcal{C}_2 - B).$$
\subsection{The Amount of Power Saving}\label{sec48}
In this subsection, we are going to compute the amount of power saving error away from the secondary term in the order of $X^{5/6m}$ when $p>5$ and the amount of power saving from the main term for $p = 3, 5$. 

Recall that in section $4.5$, we have specified the exponent of $X$ in the error term to be the maximum value among (\ref{insedl}) and (\ref{inseds}), therefore the amount of power saving is
$$\frac{5}{6m} - \frac{2}{3m}  - a(\beta, \gamma) - \epsilon - (\frac{1}{3m}-a(\beta, \gamma)-\epsilon)\cdot \max_{i,j} \{ \frac{e(i,j) +1}{\delta(i,j)}\},$$
and 
$$\frac{5}{6m} - \frac{1}{m} -  (\frac{1}{3m}-a(\beta, \gamma)-\epsilon)\cdot \max_{i,j} \{ \frac{d(i,j) +1}{\delta(i,j)}\}.$$
Recall $a = b = 4/5$ is the proved dependency of error for $S_3$ extensions. For $p>7$, we can compute the amount of power saving is
$$ \delta = \frac{1}{6m}\cdot \min\{ \frac{10\Delta/3 -a -11/6}{1/6+ a+2\Delta/3} , \frac{5\Delta/3-b}{2+b+\Delta/3} \} -\epsilon = \frac{1}{6m} \cdot \frac{5\Delta/3-b}{2+b+\Delta/3} -\epsilon,$$
%
where for the second equality we use that $a = 4/5$ and $b = 4/5$. For $p = 7$, the amount of power saving is $\delta =23/(1254m) - \epsilon$. For $p = 5$, the amount of power saving is $\delta =322/(2061m) - \epsilon$. For $ p =3$, the amount of power saving is $\delta =24/(283m) - \epsilon$.

\section{Acknowledgement}
I am extremely grateful to my advisor Melanie Matchett Wood for many helpful discussions. I would like to thank Frank Thorne for helpful conversations, thank Manjul Bhargava, Takashi Taniguchi and Frank Thorne for providing a preprint of their work, and thank Robert Harron for helpful communications and references. I would also like to thank David Roberts, Takashi Taniguchi and Frank Thorne for suggestions on an earlier draft. This work is partially supported by National Science Foundation grant DMS-$1301690$ and Vilas Early Career Investigator Award.

\Addresses
\end{document}